\documentclass[a4paper,english,fontsize=11pt,parskip=half,abstracton]{scrartcl}
\usepackage{babel}
\usepackage[utf8]{inputenc}
\usepackage[T1]{fontenc}
\usepackage[a4paper,left=25mm,right=25mm,top=30mm,bottom=30mm]{geometry}
\usepackage{amsmath}
\usepackage{amsthm}
\usepackage{amssymb}
\usepackage{enumerate}
\usepackage{aliascnt}
\usepackage[bookmarks=false,
            pdftitle={Pseudo Frobenius numbers},
            pdfauthor={Benjamin Sambale},
            pdfkeywords={Sylow's theorem},
            pdfstartview={FitH}]{hyperref}

\newtheorem*{ThmA}{Theorem A}
\newtheorem*{CorB}{Corollary B}
\newtheorem{Thm}{Theorem} 
\newaliascnt{Prop}{Thm}
\newtheorem{Prop}[Prop]{Proposition}
\aliascntresetthe{Prop}

\numberwithin{equation}{section}

\renewcommand{\phi}{\varphi}
\newcommand{\C}{\operatorname{C}}
\newcommand{\N}{\operatorname{N}}

\newcommand{\Sym}{\operatorname{Sym}}
\newcommand{\GL}{\operatorname{GL}}
\newcommand{\Irr}{\operatorname{Irr}}
\newcommand{\Syl}{\operatorname{Syl}}

\mathchardef\ordinarycolon\mathcode`\:  
 \mathcode`\:=\string"8000
 \begingroup \catcode`\:=\active
   \gdef:{\mathrel{\mathop\ordinarycolon}}
 \endgroup

\title{Pseudo Frobenius numbers}
\author{Benjamin Sambale\footnote{Fachbereich Mathematik, TU Kaiserslautern, 67653 Kaiserslautern, Germany, 
\href{mailto:sambale@mathematik.uni-kl.de}{sambale@mathematik.uni-kl.de}}}
\date{\today}

\begin{document}
\frenchspacing
\maketitle
\begin{abstract}\noindent
For a prime $p$, we call a positive integer $n$ a Frobenius $p$-number if there exists a finite group with exactly $n$ subgroups of order $p^a$ for some $a\ge 0$. Extending previous results on Sylow's theorem, we prove in this paper that every Frobenius $p$-number $n\equiv 1\pmod{p^2}$ is a Sylow $p$-number, i.\,e., the number of Sylow $p$-subgroups of some finite group. As a consequence, we verify that $46$ is a pseudo Frobenius $3$-number, that is, no finite group has exactly $46$ subgroups of order $3^a$ for any $a\ge 0$.
\end{abstract}

\section{Introduction}

Motivated by Sylow's famous theorem in finite group theory, we investigated \emph{pseudo Sylow $p$-numbers} in a previous paper~\cite{SambaleSylow}. These are positive integers $n\equiv 1\pmod{p}$, where $p$ is a prime, such that no finite group has exactly $n$ Sylow $p$-subgroups. It is known that such numbers exist whenever $p$ is odd and we gave an elementary argument for $p=17$ and $n=35$.

The present paper is based on Frobenius' extension~\cite{FrobeniusSylow2} of Sylow's theorem:

\begin{Thm}[\textsc{Frobenius}]\label{frob}
Let $p$ be a prime and $a\ge 0$ such that $p^a$ divides the order of a finite group $G$. Then the number of subgroups of order $p^a$ of $G$ is congruent to $1$ modulo $p$.
\end{Thm}

An elementary proof of \autoref{frob} has been given by Robinson~\cite{RobinsonSylow}. It is a natural question to ask if every positive integer $n\equiv 1\pmod{p}$ is a \emph{Frobenius $p$-number}, i.\,e., there exists a finite group with exactly $n$ subgroups of order $p^a$ for some $a\ge 0$. The following refinement of Frobenius' theorem, proved by Kulakoff~\cite{Kulakoff} for $p$-groups and extended to arbitrary finite groups by P.~Hall~\cite{HallFrobenius}, shows that most pseudo Sylow $p$-numbers cannot be Frobenius $p$-numbers.

\begin{Thm}[\textsc{Kulakoff--Hall}]\label{KH}
Let $p$ be a prime and $a\ge 0$ such that $p^{a+1}$ divides the order of a finite group $G$. Then the number of subgroups of order $p^a$ of $G$ is congruent to $1$ or to $1+p$ modulo $p^2$. 
\end{Thm}

The proof of \autoref{KH} uses only elementary group theory, but it lies somewhat deeper than \autoref{frob} (Kulakoff~\cite{KulakoffMiller} pointed out some errors in an earlier proof attempt by Miller~\cite{MillerSylow}). 

In view of \autoref{KH}, we call $n$ a \emph{pseudo Frobenius $p$-number} if $n$ is congruent to $1$ or $1+p$ modulo $p^2$ and no finite group has exactly $n$ subgroups of order $p^a$ for any $a\ge 0$. Obviously, every pseudo Frobenius $p$-number is a pseudo Sylow $p$-number. Since we know from \cite{SambaleSylow} that every odd number is a Sylow $2$-number, it is clear that there are no pseudo Frobenius $2$-numbers. 

Our aim in this paper is to establish the existence of a pseudo Frobenius number.
The first choices are $n=1+p$ and $n=1+p^2$. However, it can be seen that the general linear group $G=\GL_2(p^a)$ has exactly $1+p^a$ Sylow $p$-subgroups for any $a\ge 1$ (the upper unitriangular matrices form a Sylow $p$-subgroup of $G$ and the corresponding normalizer is the Borel subgroup consisting of all upper triangular matrices). 
The next candidate is $n=1+p+p^2$, but this is clearly the number of subgroups of order $p$ in the elementary abelian group $G$ of order $p^3$ (every nontrivial element of $G$ generates a subgroup of order $p$ and two distinct subgroups intersect trivially). Now for $p=3$ we might consider $n=1+2\cdot 3^2=19$. However, $19$ is a prime and we know already from \cite{SambaleSylow} that for any prime $n\equiv 1\pmod{p}$ there exist (solvable affine) groups with exactly $n$ Sylow $p$-subgroups. Finally, we have mentioned in \cite{SambaleSylow} (proved by M.~Hall~\cite{MHall}) that $n=1+3+2\cdot 3^2=22$ \emph{is} a pseudo Sylow $3$-number. On the other hand, the number of subgroups of order $9$ in the abelian group $C_9\times C_3\times C_3$ is $22$ and therefore, $22$ is not a pseudo Frobenius number. (In general, the number of subgroups of a given isomorphism type in an abelian $p$-group is given by a \emph{Hall polynomial}.)

Our first theorem in this paper deals with the case $n\equiv 1\pmod{p^2}$.

\begin{ThmA}
Every Frobenius $p$-number $n\equiv 1\pmod{p^2}$ is a Sylow $p$-number. 
\end{ThmA}

While our proof is elementary, it relies implicitly on the complicated classification of the finite simple groups (CFSG for short in the following). As an application we obtain our first pseudo Frobenius number.

\begin{CorB}
The integer $46$ is a pseudo Frobenius $3$-number.
\end{CorB}

With the examples mentioned above, it can be seen that $46$ is in fact the smallest pseudo Frobenius number.
We do not know if there are any pseudo Frobenius $p$-numbers congruent to $1+p$ modulo $p^2$. 
There are no such numbers below $100$ as one can check with the computer algebra system GAP~\cite{GAP48} for instance.

\section{Proofs}

In this section, $G$ always denotes a finite group with identity $1$ and $p$ is a prime number.
The proof of Theorem~A relies on the following more precise version of \autoref{KH} for odd primes (see \cite[Lemma~4.61 and Theorem~4.6]{HallFrobenius}).

\begin{Prop}\label{noncyclic}
Let $P$ be a Sylow $p$-subgroup of $G$ for some $p>2$. Then for $1<p^a<|P|$, the number of subgroups of order $p^a$ in $G$ is congruent to $1$ modulo $p^2$ if and only if $P$ is cyclic. 
\end{Prop}

\autoref{noncyclic} does not hold for $p=2$. For instance, the dihedral group of order $8$ (i.\,e., the symmetry group of the square) has $5\equiv 1\pmod{4}$ subgroups of order $2$ (generated by the four reflections and the rotation of degree $\pi$). A precise version for $p=2$ can be found in Murai~\cite[Theorem~D]{Muraipgrp}.

Our second ingredient is a consequence of the CFSG by Blau~\cite{Blau}.

\begin{Prop}[\textsc{Blau}]\label{blau}
If the simple group $G$ has a cyclic Sylow $p$-subgroup, then every two distinct Sylow $p$-subgroups of $G$ intersect trivially.
\end{Prop}

\begin{proof}[Proof of Theorem~A]
We may assume that $p$ is odd.
Let $n\equiv 1\pmod{p^2}$ be a minimal counterexample. Then there exists a group $G$ of minimal order such that the number of subgroups of order $p^a$ for some $a\ge 0$ is $n$. Since obviously $n>1$, we have $a\ge 1$. Moreover, since $n$ is not a Sylow $p$-number, it follows that $p^{a+1}$ divides $|G|$. By \autoref{noncyclic}, $G$ has a cyclic Sylow $p$-subgroup $P$. Since every Sylow $p$-subgroup contains exactly one subgroup of order $p^a$, the subgroups $Q=Q_1,\ldots,Q_n\le G$ of order $p^a$ form a conjugacy class in $G$. Furthermore, the number of Sylow $p$-subgroups must be greater than $n$ and this implies that some $Q_i$ is contained in two distinct Sylow $p$-subgroups. Hence by \autoref{blau}, $G$ is not simple. 

Thus, let $N$ be a nontrivial proper normal subgroup of $G$. Let $n_1$ be the number of subgroups of order $p^a$ in $QN$ (note that this number does not depend on the choice of $Q$, since every $Q_iN$ is conjugate to $QN$). Since $PN/N$ is a cyclic Sylow $p$-subgroup of $G/N$, every subgroup of order $|QN/N|$ in $G/N$ is of the form $Q_iN/N$ for some $i$. We denote the number of these subgroups by $n_2$ and conclude that $n=n_1n_2$. By construction, $n_1$ and $n_2$ are Frobenius $p$-numbers.

Suppose that $n_i\not\equiv 1\pmod{p^2}$ for some $i\in\{1,2\}$. Then $n_1\not\equiv 1\not\equiv n_2\pmod{p^2}$, since $n_1n_2= n\equiv 1\pmod{p^2}$. By \autoref{noncyclic}, $Q$ must be a Sylow $p$-subgroup of $QN$, that is 
\begin{equation}\label{equ}
|QN:Q|\not\equiv 0\pmod{p}.
\end{equation}
Similarly, $QN/N\in\Syl_p(G/N)$ or $QN/N=1$ according to \autoref{noncyclic}. In the latter case, $N$ contains $Q_1,\ldots,Q_n$ since they are all conjugate to $Q$. However, this contradicts the minimality of $G$. Hence, $QN/N$ is a Sylow $p$-subgroup of $G/N$ and $|G:QN|=|G/N:QN/N|\not\equiv 0\pmod{p}$. In combination with \eqref{equ}, we obtain 
\[|G:Q|=|G:QN||QN:Q|\not\equiv 0\pmod{p}.\]
But this contradicts the observation that $p^{a+1}$ divides $|G|$.

Consequently, $n_1\equiv n_2\equiv1 \pmod{p^2}$. The minimal choice of $G$ yields $n_2<n$. Similarly, $n_1=n$ implies $G=QN$. In this case, $P=QN\cap P=Q(N\cap P)$ (modular law) and since $P$ is cyclic we even have $Q\subseteq N\cap P\subseteq N$ and $G=QN=N$, another contradiction. Thus, $n_1<n$. Since $n$ is a minimal counterexample to our theorem, $n_1$ and $n_2$ must be Sylow $p$-numbers, since they are Frobenius $p$-numbers. Let $H_i$ be a finite group with exactly $n_i$ Sylow $p$-subgroups for $i=1,2$. Then 
\[\Syl_p(H_1\times H_2)=\{S_1\times S_2:S_i\in\Syl_p(H_i)\}\] 
and $n=n_1n_2$ is a Sylow $p$-number (of $H_1\times H_2$). This final contradiction completes the proof.
\end{proof}

As in the previous paper~\cite{SambaleSylow}, we make use of the first principles of group actions. Recall that an \emph{action} of $G$ on a finite nonempty set $\Omega$ is a map $G\times\Omega\to\Omega$, $(g,\omega)\mapsto{^g\omega}$ such that $^1\omega=\omega$ and $^g({^h\omega})={^{gh}\omega}$ for $g,h\in G$ and $\omega\in\Omega$. Every action gives rise to a homomorphism $G\to\Sym(\Omega)$ into the symmetric group on $\Omega$, and the action is called \emph{faithful} whenever this homomorphism is injective. In this case $G$ is a \emph{permutation group} of \emph{degree} $|\Omega|$. 
The \emph{orbit} of $\omega\in\Omega$ under $G$ is the subset $^G\omega:=\{^g\omega:g\in G\}\subseteq\Omega$. 
The \emph{orbit-stabilizer theorem} states that 
\[|^G\omega|=|G:G_\omega|\] 
where $G_\omega:=\{g\in G:{^g\omega=\omega}\}$ is the \emph{stabilizer} of $\omega\in\Omega$.
We say that $G$ acts \emph{transitively} on $\Omega$ if there is only one orbit, i.\,e., $\Omega={^G\omega}$ for any $\omega\in\Omega$. A subset $\Delta\subseteq\Omega$ is called a \emph{block} if $^g\Delta\cap\Delta\in\{\Delta,\varnothing\}$ for every $g\in G$. A transitive action is called \emph{primitive} if there are no blocks $\Delta$ with $1<|\Delta|<|\Omega|$. This happens if and only if $G_\omega$ is a maximal subgroup of $G$ for any $\omega\in\Omega$. Finally, a transitive action is \emph{$2$-transitive} if $G_\omega$ acts transitively on $\Omega\setminus\{\omega\}$ for any $\omega\in\Omega$. 
In the following we are mainly interested in the transitive conjugation action of $G$ on $\Syl_p(G)$. Here the stabilizer of $P\in\Syl_p(G)$ is the \emph{normalizer} $\N_G(P):=\{g\in G:gP=Pg\}$.

In the proof of Corollary~B we apply two further results. The first appeared in Wielandt~\cite{Wielandt2p} and was reproduced in Cameron's book~\cite[Theorem~3.25]{Cameron}.

\begin{Prop}[\textsc{Wielandt}]\label{wielandt}
Let $G$ be a primitive permutation group of degree $2p$. Then $G$ is $2$-transitive or $2p-1$ is a square.
\end{Prop}

It is another consequence (which we do not need) of the CFSG that the second alternative in \autoref{wielandt} only occurs for $p=5$. 

The second tool for Corollary~B is a consequence of Brauer's theory of $p$-blocks of defect $1$~\cite{Brauerdef11} and can be extracted from Navarro's book~\cite[Theorem~11.1]{Navarro}. Here, $\Irr(G)$ is the set of irreducible complex characters of $G$ and the trivial character is denoted by $1_G$.

\begin{Prop}[\textsc{Brauer}]\label{brauer}
Suppose that $G$ has a Sylow $p$-subgroup $P$ of order $p$ such that $\C_G(P)=P$ and $e:=\lvert\N_G(P)/P\rvert$. Then there exists a set of irreducible characters 
\[B=\{1_G=\chi_1,\ldots,\chi_e,\psi_1,\ldots,\psi_{(p-1)/e}\}\subseteq\Irr(G)\] 
and signs $\epsilon_1,\ldots,\epsilon_e\in\{\pm1\}$ such that
\begin{align*}
\chi_i(1)&\equiv\epsilon_i\pmod{p}&&(1\le i\le e),\\
\psi_j(1)&=\Bigl|\sum_{i=1}^e\epsilon_i\chi_i(1)\Bigr|&&(1\le j\le (p-1)/e),\\
\mu(1)&\equiv 0\pmod{p}&&(\forall\mu\in\Irr(G)\setminus B).
\end{align*}
\end{Prop}

The special case $e=1$ in \autoref{brauer} leads to $1=1_G(1)=\chi_1(1)=\psi_1(1)=\ldots=\psi_{p-1}(1)$ and $|G:G'|=p$ where $G'$ is the commutator subgroup of $G$ (see \cite[Problem~15.6]{IsaacsAlgebra}). 
In general, \autoref{brauer} provides information on $|G|$, because it is known that the irreducible character degrees divide the group order (see \cite[Problem~28.12]{IsaacsAlgebra}). 

Recall that every action of $G$ on $\Omega$ gives rise to a \emph{permutation character} $\pi$ which counts the number of fixed points, that is, $\pi(g):=|\{\omega\in\Omega:{^g\omega=\omega}\}|$ for $g\in G$ (see \cite[Section~2.5]{Cameron}). The action is $2$-transitive if and only if $\pi=1_G+\chi$ for some $\chi\in\Irr(G)\setminus\{1_G\}$.

\begin{proof}[Proof of Corollary~B]
By Theorem~A, it suffices to show that $46$ is a pseudo Sylow $3$-number, because $46\equiv 1\pmod{9}$. 
Let $G$ be a minimal counterexample such that $\lvert\Syl_3(G)\rvert=46$. By Sylow's theorem, $G$ acts transitively on $\Syl_3(G)$. If $K\unlhd G$ is the kernel of this action, then it is easy to see that $G/K$ has the same number of Sylow $3$-subgroups (see \cite[Step~1 of proof of Theorem~A]{SambaleSylow}). Thus, by minimality $K=1$ and $G$ acts faithfully on $\Syl_3(G)$. In particular, we can and will regard $G$ as a subgroup of the symmetric group $S_{46}$. Then, every Sylow $3$-subgroup lies in the alternating group $A_{46}$ and minimality implies $G\le A_{46}$.
For $P\in\Syl_3(G)$ let $\N_G(P)<M\le G$. Then $P\in\Syl_3(M)$ and 
\[\lvert\Syl_3(M)\rvert=|M:\N_M(P)|=|M:\N_G(P)|\in\{2,23,46\}\] 
by Lagrange's theorem. Since $2$ and $23$ are not congruent to $1$ modulo $3$, we must have $M=G$. Hence, $\N_G(P)$ is a maximal subgroup of $G$ and therefore $G$ acts primitively on $\Syl_3(G)$. 

[At this point we could refer to the database of primitive permutation groups of small degree (see for instance Dixon--Mortimer~\cite[Appendix~B]{DM} or \cite{GAP48,oeis2}). However, this database is based on the Aschbacher--O'Nan--Scott theorem and relies ultimately on the CFSG. We prefer to give a classification-free argument along the lines of M.~Hall's paper~\cite{MHall}.]

Since $45$ is not a square, \autoref{wielandt} implies that $G$ acts $2$-transitively on $\Syl_3(G)$, i.\,e., $\N_G(P)$ acts transitively on $\Syl_3(G)\setminus\{P\}$. Hence, for $Q\in\Syl_3(G)\setminus\{P\}$, the 2-point stabilizer $\N_G(P)\cap\N_G(Q)$ has index $45$ in $\N_G(P)$ by the orbit-stabilizer theorem. Since $\N_P(Q)$ is a Sylow $3$-subgroup of $\N_G(P)\cap\N_G(Q)$, the orbit $^PQ$ of $P$ has size
\[|^PQ|=|P:\N_P(Q)|=9.\]
For $g\in\N_G(P)$ we have 
\[^g({^PQ})={^{gP}Q}={^{Pg}Q}={^P({^gQ})}.\] 
Since the orbits of $P$ are disjoint, $^PQ$ is a block of $\N_G(P)$. Since $\N_G(P)$ is transitive on $\Syl_3(G)\setminus\{P\}$, the distinct conjugates of $^PQ$ form a partition of $\Syl_3(G)\setminus\{P\}$ into five blocks with nine points each. Moreover, $\N_G(P)$ permutes these blocks. Suppose that there exists an element $x\in\N_G(P)$ of order $11$. Then $x$ must fix each of the five blocks. On the other hand, $x$ cannot permute nine points nontrivially. Hence, $x$ cannot exist and by Cauchy's theorem, $|G|=46\lvert\N_G(P)\rvert$ is not divisible by $11$. 
Similarly, $\lvert\N_G(P)\rvert$ is not divisible by $23$. 

Now let $S\in\Syl_{23}(G)$. Then $|S|=23$ and $S$ is generated by a product of two disjoint $23$-cycles, since $\lvert\N_G(P)\rvert$ is not divisible by $23$. It follows that $\C_G(S)\le\C_{A_{46}}(S)=S$ (see \cite[Lemma~5]{SambaleSylow}). Moreover, $\lvert\N_G(S)/S\rvert$ divides $22$ (see \cite[Lemma~6]{SambaleSylow}). By Lagrange's theorem, $\lvert\N_G(S)\rvert$ is not divisible by $11$ and therefore $\lvert\N_G(S)/S\rvert\in\{1,2\}$. 
In the first case, $|G:G'|=23$ by the remark after \autoref{brauer}. 
However, this contradicts the minimal choice of $G$, since every Sylow $3$-subgroup of $G$ lies in $G'$.
Hence, $\lvert\N_G(S)/S\rvert=2$. 

The permutation character of our $2$-transitive group $G$ has the form $1_G+\chi$ where $\chi\in\Irr(G)$ has degree $45$ (see \cite[Section~2.5]{Cameron}). With the notation of \autoref{brauer} for $p=23$, we have $\psi_j(1)\equiv\pm2\pmod{23}$ for $j=1,\ldots,11$ and it follows that $\chi=\chi_2$, $\epsilon_2=-1$ and $\psi_1(1)=|1-45|=44$. However, the degree of every irreducible character divides the group order, but $|G|$ is not divisible by $44=4\cdot 11$. Contradiction.
\end{proof}

It is possible to prove Corollary~B directly without appealing to Theorem~A or \autoref{blau}. 
To do so, one has to study the conjugation action on the set of $46$ subgroups of a fixed $3$-power order which is still ($2$-)transitive by \autoref{noncyclic}.

Using the database of primitive permutation groups mentioned in the proof, it is easy to show that $51$ is a pseudo Frobenius $5$-number.  

\section*{Acknowledgment}
The author thanks the anonymous reviewer for pointing out a missing argument in the proof of Theorem~A.
This work is supported by the German Research Foundation (projects SA 2864/1-1 and SA 2864/3-1).

\end{document}